\documentclass[11pt,reqno]{amsart}
\pdfoutput=1

\usepackage{latexsym,amsthm,amssymb,amsmath,amscd,mathtools,hyperref,xpatch}
\mathtoolsset{showonlyrefs}
\usepackage{tikz-cd}
\tikzcdset{
  cells={font=\everymath\expandafter{\the\everymath\displaystyle}},
}
\usetikzlibrary{matrix,arrows,decorations.pathmorphing}
\usepackage{bm}
\usepackage{fancyhdr}
\usepackage{mathtools}
\usepackage{mathrsfs} 
\usepackage{tcolorbox}

\usepackage[normalem]{ulem}

\newcommand{\nospacepunct}[1]{\makebox[0pt][l]{\,#1}}

\usepackage[hmargin=2cm,vmargin=1.5cm]{geometry}

\usepackage{graphicx}

\usepackage[shortlabels]{enumitem}
\setlist[enumerate]{nosep}

\usepackage{xcolor}
\newcommand\redsout{\bgroup\markoverwith{\textcolor{red}{\rule[0.5ex]{2pt}{0.8pt}}}\ULon}

\newtheorem{theorem}{Theorem}[section]

\newtheorem*{theorem*}{Theorem}
\newtheorem{lemma}[theorem]{Lemma}
\newtheorem*{lemma*}{Lemma}
\newtheorem{corollary}[theorem]{Corollary}
\newtheorem{proposition}[theorem]{Proposition}
\newtheorem{remark}[theorem]{Remark}
\newtheorem{definition}[theorem]{Definition}
\newtheorem*{definition*}{Definition}
\newtheorem*{definitions*}{Definitions}

\newtheorem*{question*}{Question}

\newtheorem*{questions*}{Questions}

\newcommand{\cC}{\mathcal C}

\newcommand{\cE}{\mathcal E}
\newcommand{\cF}{\mathcal F}

\newcommand{\cK}{\mathcal K}

\newcommand{\cO}{\mathcal O}
\newcommand{\cP}{\mathcal P}
\newcommand{\cQ}{\mathcal Q}
\newcommand{\cR}{\mathcal R}

\newcommand{\cX}{\mathcal X}
\newcommand{\cZ}{\mathcal Z}

\def\Az{\mathbb{A}}
\def\Cz{\mathbb{C}}

\def\Kz{\mathbb{K}}

\def\Qz{\mathbb{Q}}
\def\Rz{\mathbb{R}}
\def\Tz{\mathbb{T}}
\def\Zz{\mathbb{Z}}

\def\1z{\mathbb{1}}
\newcommand{\fA}{\mathfrak A}
\newcommand{\fB}{\mathfrak B}

\def\SEMI{\mbox{$\times\kern-2pt\vrule height5pt width.6pt \kern3pt $}}

\newcommand{\id}{{\rm id}}

\newcommand{\p}{\mathfrak{p}}

\newcommand{\Prim}{\textup{Prim}}

\newcommand{\acts}{\curvearrowright} 

\renewcommand{\phi}{\varphi}
\newcommand{\eps}{\varepsilon}

\newcommand{\im}{\textup{im}}

\renewcommand{\O}{\mathcal{O}}

\newcommand{\ol}{\overline}

\newcommand{\pmone}{\{\pm 1\}}

\newcommand{\K}{\textup{K}}
\newcommand{\KK}{\textup{KK}}

\newcommand{\ho}{\hat{\otimes}}
\DeclareMathOperator{\eval}{ev}
\DeclareMathOperator{\Abs}{Abs}

\newcommand{\A}{\mathscr{A}}
\newcommand{\B}{\mathscr{B}}

\begin{document}

\title[On the C*-algebra associated with the full adele ring]{On the C*-algebra associated with the full adele ring\\ of a number field}

\thispagestyle{fancy}

\author{Chris Bruce}\thanks{C. Bruce was supported by a Banting Fellowship administered by the Natural Sciences and Engineering Research Council of Canada (NSERC) and has received funding from the European Union’s Horizon 2020 research and innovation programme under the Marie Sklodowska-Curie grant agreement No 101022531 and the European Research Council (ERC) under the European Union’s Horizon 2020 research and innovation programme (grant agreement No. 817597).}

\address[Chris Bruce]{School of Mathematics and Statistics, University of Glasgow, University Place, Glasgow G12 8QQ, United Kingdom}
\email[Bruce]{Chris.Bruce@glasgow.ac.uk}

\author{Takuya Takeishi}\thanks{T. Takeishi is supported by JSPS KAKENHI Grant Number JP19K14551.}
\address[Takuya Takeishi]{
	Faculty of Arts and Sciences\\
	Kyoto Institute of Technology\\
	Matsugasaki, Sakyo-ku, Kyoto\\
	Japan}
\email[Takeishi]{takeishi@kit.ac.jp}

\date{\today}

\maketitle
\begin{abstract} 
The multiplicative group of a number field acts by multiplication on the full adele ring of the field.
Generalising a theorem of Laca and Raeburn, we explicitly describe the primitive ideal space of the crossed product C*-algebra associated with this action. We then distinguish real, complex, and finite places of the number field using \K-theoretic invariants. Combining these results with a recent rigidity theorem of the authors implies that any *-isomorphism between two such C*-algebras gives rise to an isomorphism of the underlying number fields that is constructed from the *-isomorphism.
\end{abstract}

\setlength{\parindent}{0cm} \setlength{\parskip}{0.5cm}

\section{Introduction} 

Each number field $K$ naturally gives rise to two topological dynamics systems: The multiplicative group $K^*$ of $K$ acts by multiplication both on the full adele ring $\Az_K$ of $K$, and on the finite adele ring $\Az_{K,f}$ of $K$.
Connes defined the \emph{adele class space} of a number field $K$ to be the quotient space $\Az_K/K^*$. This space is used in Connes' reformulation of the Riemann hypothesis \cite{C}, and has since been studied in, for example, \cite{CC:11,CCM,CCM:09}. The space $\Az_K/K^*$ is quite difficult since, for instance, the action $K^*\acts \Az_K$ is ergodic \cite{LagNesh}, so the philosophy of noncommutative geometry says that one should consider the crossed product C*-algebra $\A_K:=C_0(\Az_K)\rtimes K^*$ attached to $K^*\acts\Az_K$. 

This paper is a sequel to \cite{BT}, where we considered the ``finite part'' of $\A_K$, namely the crossed product C*-algebra $\fA_K:=C_0(\Az_{K,f})\rtimes K^*$ attached to $K^*\acts \Az_{K,f}$. We studied primitive ideal spaces and subquotients of $\fA_K$ and related auxiliary C*-algebras, and we proved that this class of C*-algebras is rigid in a strong sense, see \cite[Theorem~1.1]{BT}.
Here, we study the primitive ideal space and subquotients of $\A_K$ and of the auxiliary C*-algebra $\B_K:=C_0(\Az_K/\mu_K)\rtimes \Gamma_K$, where $\mu_K$ is the group of roots of unity in $K$ and $\Gamma_K:=K^*/\mu_K$.
For context and references on C*-algebras of number-theoretic origin, we refer the reader to the introduction of \cite{BT}.
The study of $\A_K$ is complicated by the fact that $\Az_K$ has a large connected component, whereas $\Az_{K,f}$ is totally disconnected. However, the C*-algebra $\A_K$ is one of the most natural C*-algebras associated with a number field. For the number field $K=\Qz$, Laca and Raeburn computed the primitive ideal space $\Prim(\A_\Qz)$ of the C*-algebra $\A_\Qz$ \cite{LR:00}. This space has a rich structure involving the idele class group $C_\Qz$ of $\Qz$, the power set of the primes $2^{\cP_\Qz}$, and the compact group $\widehat{\Qz^*}$. Laca and Raeburn did not address the natural problem of computing $\Prim(\A_K)$ for a general number field $K$, and our first achievement is to give the computation of $\Prim(\A_K)$ for any number field $K$ (Theorem~\ref{thm:basicopensets} and Proposition~\ref{prop:primA}). 
Although the statements from \cite{LR:00} generalise easily, the proofs become much more technical in several places, essentially due to the fact that the idele class group of a general number field does not have a canonical decomposition as a product of a connected group and a totally disconnected group. 
Our results include an explicit basis of open sets for the quasi-orbit space over which $\Prim(\A_K)$ sits, which even in the rational case offers a mild improvement on \cite{LR:00}.

The main achievement of this paper (Theorem~\ref{thm:main}) is a K-theoretic distinction of real, complex, and finite places of $K$ in terms of boundary maps between \K-groups of subquotients of the auxiliary C*-algebra $\B_K$: Associated with each place $v$ of $K$, there is an extension $\cE_K^v$ of subquotients of $\B_K$. We prove that the boundary map $\partial_K^v$ of $\cE_K^v$ is bijective if and only if $v$ is real, is injective but not surjective if and only if $v$ is complex, and is not injective if and only if $v$ is finite. 
The place $v$ also canonically determines a composition factor $\A_K^v$ of $\A_K$, and the aforementioned properties of the boundary maps $\partial_K^v$ imply that the type of the place $v$ is determined by \K-theoretic invariants. Combining this with our description of $\Prim(\A_K)$, we deduce the following consequence: Any *-isomorphism $\alpha\colon\A_K\overset{\cong}{\to}\A_L$ descends to a *-isomorphism $\bar{\alpha}\colon \fA_K\overset{\cong}{\to}\fA_L$. Therefore, by \cite[Theorem~1.1]{BT}, we obtain that $\A_K$ and $\A_L$ are isomorphic if and only if $K$ and $L$ are isomorphic (Corollary~\ref{cor:main}). Moreover, the field isomorphism $K\cong L$ coming from a *-isomorphism $\alpha\colon\A_K\overset{\cong}{\to}\A_L$ is constructed from $\alpha$ as in \cite[Theorem~1.1]{BT}.
It is possible to distinguish infinite and finite places by looking only at the \K-groups of composition factors themselves (Remark~\ref{rmk:ktheory}), and this is sufficient to derive Corollary~\ref{cor:main}. The point here is that we can detect real and complex places via the boundary maps, which is interesting in its own right.
This is the first (complete) rigidity result in the setting of C*-algebras of number-theoretic origin where the C*-algebra is built from an action on a space that is not zero-dimensional.

The paper is organised as follows. Some notation and preliminaries are given in Section~\ref{sec:prelims}. Section~\ref{sec:prims} contains our computation of $\Prim(\A_K)$, and some discussion of extensions and subquotients associated with finite sets of places is given in Section~\ref{sec:subquotients}. Section~\ref{sec:main} contains our main theorem on \K-theoretic distinction of places.

\textbf{Acknowledgement.}
We thank Yuki Arano for helpful comments. 
The final stage of this research was carried out during a visit of C. Bruce to the Kyoto Institute of Technology, which was partially funded by an LMS Early Career Researcher Travel Grant
(Ref: ECR-2122-53); he would like to thank the mathematics department at KIT and the LMS for their support.

\section{Preliminaries}
\label{sec:prelims}
Let $K$ be a number field with ring of integers $\cO_K$, and let $\Sigma_K$ be the set of places of $K$. For $v\in\Sigma_K$, we write $v\mid\infty$ (resp. $v\nmid \infty$) if $v$ is infinite (resp. finite), and we let $\Sigma_{K,f}$ and $\Sigma_{K,\infty}$ denote the sets of finite and infinite places, respectively, so that $\Sigma_K=\Sigma_{K,f}\sqcup\Sigma_{K,\infty}$. 
The absolute value associated with $v\in\Sigma_K$ will be denoted by $|\cdot|_v$, and we shall use $K_v$ to denote the locally compact completion of $K$ with respect to $|\cdot|_v$; when $v$ is finite, we let $\cO_{K,v}$ be the compact ring of integers in $K_v$, and when $v$ is infinite, we put $\cO_{K,v}:=K_v$. 
Then, the full (resp. finite) adele ring of $K$ is the restricted product  $\Az_K:=\prod_{v\in\Sigma_K}'(K_v,\cO_{K,v})$ (resp. $\Az_{K,f}:=\prod_{v\in \Sigma_{K,f}}'(K_v,\cO_{K,v})$). Each $a\in\Az_K$ is of the form $a=(a_\infty,a_f)$, where $a_\infty\in \prod_{v\in\Sigma_K,\infty}K_v$ and $a_f\in\Az_{K,f}$.
Given $a=(a_w)_w\in\Az_K$, and $v\in\Sigma_K$, we let $|a|_v:=|a_v|_v$.
We let $\Az_K^*:=\prod_{v\in \Sigma_K}'(K_v^*,\cO_{K,v}^*)$ be the idele group of $K$ equipped with the restricted product topology, and similarly define $\Az_{K,f}^*:=\prod_{v\in \Sigma_{K,f}}'(K_v^*,\cO_{K,v}^*)$. Then, $K^*:=K\setminus\{0\}$ embeds diagonally into $\Az_K^*$, and the idele class group of $K$ is the quotient $C_K:=\Az_K^*/K^*$.
Let $\cP_K$ be the set of nonzero prime ideals of $\cO_K$. There is canonical bijection $\cP_K\cong \Sigma_{K,f}$ which we denote by $\p\mapsto v_\p$. For $\p\in\cP_K$, we let $|\cdot |_\p:=|\cdot|_{v_\p}$, $K_\p:=K_{v_\p}$ and $\cO_{K,\p}:=\cO_{K,v_\p}$.

\begin{definition}
Let $K$ be a number field, and let $\mu_K$ denote the group of roots of unity in $K$. Put $\Gamma_K:=K^*/\mu_K$. The group $K^*$ acts on $\Az_K$ by multiplication through the diagonal embedding $K^*\hookrightarrow \Az_K$, and this action descends to an action of $\Gamma_K$ on $\Az_K/\mu_K$.
We let 
\[
\A_K:=C_0(\Az_K)\rtimes K^*,
\]
and 
\[
\B_K:=C_0(\Az_K/\mu_K)\rtimes \Gamma_K.
\]
be the associated C*-algebra crossed products.
\end{definition}

The groups $K^*$ and $\Gamma_K$ also act canonically on $\Az_{K,f}$ and $\Az_{K,f}/\mu_K$, respectively. As in \cite{BT}, we let $\fA_K:=C_0(\Az_{K,f})\rtimes K^*$ and $\fB_K:=C_0(\Az_{K,f}/\mu_K)\rtimes\Gamma_K$.

The power-cofinite topology on $2^{\Sigma_K}$ is the topology whose basic open sets are of the form 
\[
U_F:=2^{\Sigma_K\setminus F}=\{T\subseteq\Sigma_K : T\cap F=\emptyset\}
\]
as $F$ ranges over the finite subsets of $\Sigma_K$.

We shall need several notions on C*-algebras over topological spaces. We refer the reader to \cite{Kirchberg,MeyerNest:09} for the general theory and to \cite[\S~4.1~and~4.2]{BT} for a discussion and results tailored to our situation.

Given a C*-algebra $A$, we let $\K_i(A)$ denote the \K-groups of $A$ for $i\in\{0,1\}$, and we put $\K_*(A)=\K_0(A)\oplus \K_1(A)$, which we view as a $\Zz/2\Zz$-graded abelian group. For a short exact sequence of C*-algebras
\[
\cE\colon 0\to B\to E\to A\to 0,
\]
we let $\partial_\cE$ denote the boundary map $\partial_\cE\colon \K_*(A)\to K_{*+1}(B)$, and for a nuclear C*-algebra $D$, we let $\cE\otimes D$ denote the extension 
\[
\cE\otimes D\colon 0\to B\otimes D\to E\otimes D\to A\otimes D\to 0.
\]

\section{The primitive ideal space}
\label{sec:prims}
In this section, we generalise the main results from \cite[\S~3]{LR:00} to arbitrary number fields. 
We first describe the quasi-orbit space for the action $K^*\acts\Az_K$. For background on quasi-orbit spaces, see \cite[Chapter~6]{Will:book} (cf. \cite[\S~4.3]{BT} for a discussion  tailored to our situation).

For $a=(a_v)_v\in\Az_K$, let $\cZ(a):=\{v\in\Sigma_ K: a_v=0\}$.

\begin{proposition}[{cf. \cite[Lemmas~3.1~and~3.2]{LR:00}}]
\label{prop:quasi-orbits}
If $a=(a_v)_v\in\Az_K$, then
\[
\overline{K^*a}=\begin{cases}
K^*a & \text{ if } a\in \Az_K^*,\\
\{b\in\Az : \cZ(a)\subseteq \cZ(b) \} & \text{ if } a\in\Az_K\setminus \Az_K^*.
\end{cases}
\]
\end{proposition}

\begin{proof}
If $a \in \Az_K^*$, then $K^*a$ is closed in $\Az_K$ since $K \subseteq \Az_K$ is discrete (see \cite[p.64]{CF}), so that $\ol{K^*a}=K^*a$. If $a \in \Az_K \setminus \Az_K^*$ and $\cZ(a) = \emptyset$, then $K^*a$ is dense in $\Az_K$ by \cite[\S~5.1]{LagNesh}, so that the claim holds in this case. Finally, if $a \in \Az_K$ and $\cZ(a) \neq \emptyset$, then by Strong Approximation, $K$ is dense in $\prod_S' (K_v, \cO_{K,v})$, where $S= \Sigma_K \setminus \cZ(a)$ (see \cite[p.67]{CF}), so a similar argument as in \cite[Lemma~4.10]{BT} verifies the claim. 
\end{proof}

\begin{definition} \label{def:quasi-orbit}
Let
\begin{equation}
\label{eqn:qi}
\Xi\colon \Az_K\to 2^{\Sigma_K}\sqcup C_K,\quad  a\mapsto\begin{cases}
\cZ(a) & \text{ if } a\in\Az_K\setminus\Az_K^*,\\
aK^* & \text{ if } a\in\Az_K^*.
\end{cases}
\end{equation}
\end{definition}
By Proposition~\ref{prop:quasi-orbits}, the map $\Xi$ descends to a bijection from $\cQ(\Az_K/K^*)$ onto $2^{\Sigma_K}\sqcup C_K$. We identify $\cQ(\Az_K/K^*)$ with $2^{\Sigma_K}\sqcup C_K$, and then $\Xi$ is identified with the quasi-orbit map.

\begin{definition}
The \emph{quasi-orbit} topology on $2^{\Sigma_K}\sqcup C_K$ is the topology induced by the map $\Xi$ from Definition~\ref{def:quasi-orbit}.
We let ${\rm qo}$ denote the quasi-orbit topology on $2^{\Sigma_K}\sqcup C_K$, and we let $\rm pc$  denote the power-cofinite topology on $2^{\Sigma_K}$.
\end{definition}

Since the quasi-orbit map $\Az_K\to \cQ(\Az_K/K^*)$ is open, $\Xi$ is an open map by definition of the topology on $2^{\Sigma_K}\sqcup C_K$. Our aim now is to describe this topology, which will be the generalisation of \cite[Theorem~3.7]{LR:00} to number fields. The subtlety comes down to the distinction between the adelic and idelic topologies on subsets of $\Az_K^*$. The proofs for general number fields are complicated because one does not have an explicit description of $C_K$ in general, whereas $C_\Qz\cong \Rz_+^*\times\prod_{p\in\cP_\Qz} \Zz_p^*$.

It will be convenient to introduce some terminology for topologies on subsets of $\Az_K^*$ and $C_K$.
\begin{definition}
Let $A\subseteq \Az_K^*$. The \emph{idelic topology} (resp. \emph{adelic topology}) on $A$ is the subspace topology on $A$ when viewed as a subspace of $\Az_K^*$ (resp. of $\Az_K$). We shall also call the quotient topology on $C_K$ the \emph{idelic topology}.
Let ${\rm ad}$ and ${\rm id}$ denote the adelic topology and idelic topology, respectively. 
\end{definition}
Note that the idelic topology on a subset $A\subseteq\Az_K^*$ is always finer than the adelic topology. Given a topological space $(X,\tau)$ and a subset $A\subseteq X$, we let $\ol{A}^\tau$ denote the closure of $A$ in the $\tau$ topology.

\begin{proposition}
\label{prop:qipc}
For $A\subseteq 2^{\Sigma_K}$, we have 
\[
\ol{A}^{\rm qo}=\begin{cases}
\ol{A}^{\rm pc} & \text{ if } \ol{A}^{\rm pc} \subsetneqq 2^{\Sigma_K},\\
2^{\Sigma_K}\sqcup C_K & \text{ if } \ol{A}^{\rm pc} = 2^{\Sigma_K}.
\end{cases}
\]
In particular, the relative topology on $2^{\Sigma_K}$ as a subspace of $2^{\Sigma_K}\sqcup C_K$ is the same as the power-cofinite topology.
\end{proposition}
\begin{proof}
The proof is essentially the same as the case $K=\Qz$ from \cite[Theorem~3.7]{LR:00}. First, we show that the inclusion map $(2^{\Sigma_K},\, {\rm pc}) \to (2^{\Sigma_K}\sqcup C_K,\, {\rm qo})$ is continuous.  
It suffices to show that for any basic open set of $\Az_K$ of the form $U = \prod_{\Sigma_K} U_v$, where $U_v$ is an open set of $K_v$ and $U_v= \cO_{K,v}$ for all but finitely many $v \in \Sigma_K$, the set $\Xi(U) \cap 2^{\Sigma_K}$ is power-cofinite open. In fact, by choosing appropriate non-invertible adeles from $U$, we see that $\Xi(U) \cap 2^{\Sigma_K} = U_F$, where $F=\{v \in \Sigma_K \colon 0 \not\in U_v\}$. 
Note that this argument also shows both that $\emptyset \in \Xi(U)$ for such $U$, which implies that $\{\emptyset\}$ is dense in $2^{\Sigma_K}\sqcup C_K$, and that the identity map $(2^{\Sigma_K},\, {\rm pc}) \to (2^{\Sigma_K},\, {\rm qo})$ is an open map and hence is a homeomorphism, which settles the last claim.

Let $A\subseteq 2^{\Sigma_K}$. 
For a non-empty finite subset $F$ of $\Sigma_K$, the set $2^{\Sigma_K} \setminus U_F$ is closed in $2^{\Sigma_K}\sqcup C_K$, since $\Xi(U) = U_F \sqcup C_K$ for $U = \prod_F K_v^* \times \prod_{F^c} \cO_{K,v}$. Hence, we have $\ol{A}^{\rm qo} = \ol{A}^{\rm pc}$ if $\ol{A}^{\rm pc} \neq 2^{\Sigma_K}$, since $A$ is contained in $2^{\Sigma_K} \setminus U_F$ for some non-empty finite $F \subseteq \Sigma_K$. 
In addition, we have $\ol{A}^{\rm qo} = 2^{\Sigma_K}\sqcup C_K$ if $\ol{A}^{\rm pc} = 2^{\Sigma_K}$, since $\emptyset \in \ol{A}^{\rm pc} \subseteq \ol{A}^{\rm qo}$ by the continuity of $(2^{\Sigma_K},\, {\rm pc}) \to (2^{\Sigma_K}\sqcup C_K,\, {\rm qo})$.  
\end{proof}

Let $||\cdot||\colon \Az_K\to [0,\infty)$ denote the adelic norm given by $||a||:=\prod_{v\in\Sigma_K}|a|_v$. The restriction of $||\cdot||$ to $\Az_K^*$ is the usual idelic norm, which is continuous with respect to the idelic topology. We also write $||\cdot||$ for the induced norm map $C_K\to \Rz_+^*$.

\begin{definition}
For each $\eps>0$, let 
\[
\Az_K^{\eps}:=\{a\in\Az_K : ||a||\geq \eps\}.
\]
\end{definition}

\begin{remark}
  For $a\in\Az_K$, we have $||a||>0$ if and only if $a\in\Az_K^*$, so $\Az_K^{\eps}\subseteq\Az_K^*$ for every $\eps>0$. 
\end{remark}

We shall need the following modification of a standard number-theoretic fact:

\begin{lemma}[cf. {\cite[p.69~and~p.70]{CF}}]\label{lem:coincide}
For each $\eps>0$, the idelic and adelic topologies on $\Az_K^{\eps}$ coincide. 
\end{lemma}
\begin{proof} 
Since the idelic topology is finer, we only need to show that given an $a\in \Az_K^{\eps}$ and an open set $V\subseteq \Az_K^*$ in the idelic topology with $a\in V$, that there exists an open set $U\subseteq \Az_K$ in the adelic topology such that $a\in U$ and $U\cap \Az_K^{\eps}\subseteq V\cap \Az_K^{\eps}$.
Fix $a\in \Az_K^{\eps}$, and let $V\subseteq \Az_K^*$ be an open set in the idelic topology with $a\in V$. We may assume that $V$ is of the form
\[
V=\prod_{v\in F}B(a_v,\delta_v)\times \prod_{\p\in \cP_K\setminus F}\O_\p^*,
\]
where $F\subseteq \Sigma_K$ is a finite subset such that $\Sigma_{K,\infty}\subseteq F$, $\delta_v\in\Rz_{>0}$, and $B(a_v,\delta_v):=\{x\in K_v:|x-a_v|_v<\delta_v\}$. 
Moreover, we may assume that $F$ is large enough so that $||a||<\eps N(\p)$ for all $\p\in \cP_K\setminus F$. Choose $\delta>0$ such that $||a||+\delta < \varepsilon N(\p)$ for all $\p\in \cP_K\setminus F$. Now choose $\delta_v'\in\Rz_{>0}$ for $v\in F$ small enough so that $\prod_{v\in F}(|a_v|_v+\delta_v')<\left(\prod_{v\in F}|a_v|_v\right)+\delta=||a||+\delta$ and $\delta_v' < \delta_v$.
Put 
\[
U:=\prod_{v\in F}B(a_v,\delta_v')\times \prod_{\p\in \cP_K\setminus F}\O_\p.
\]
Then, $U$ is open in the adelic topology and $a\in U$. Let $b\in U\cap \Az_K^{\eps}$. Then,
\[
\eps\leq ||b||< \Bigg(\prod_{v\in F}(|a_v|_v+\delta_v')\Bigg)\Bigg(\prod_{\p\in\cP_K\setminus F}|b_\p|_\p\Bigg)\leq (||a||+\delta) \Bigg(\prod_{\p\in\cP_K\setminus F}|b_\p|_\p\Bigg).
\]
Since $|b_\p|_\p\leq 1$ for all $\p\in \cP_K\setminus F$, it follows that $\eps<(||a||+\delta)|b_\p|_\p$ for all $\p\in \cP_K\setminus F$. Now we have 
\[
\frac{1}{N(\p)}<\frac{\eps}{||a||+\delta}<|b_\p|_\p\leq 1 \quad (\p\in \cP_K\setminus F),
\]
so that $|b_\p|_\p=1$ for all $\p\in \cP_K\setminus F$. Thus, $b\in V$. This shows that $U\cap \Az_K^{\eps}\subseteq V\cap\Az_K^{\eps}$.
\end{proof}

Continuity of the idelic norm map implies that $\Az_K^{\eps}$ is closed in $\Az_K^*$. We will need the following stronger statement, which was proven in \cite[Lemma~3.5]{LR:00} for the case $K=\Qz$. 

\begin{lemma}
 \label{lem:usc}
For each $\eps>0$, $\Az_K^{\eps}$ is closed in $\Az_K$ (that is, the map $||\cdot||\colon \Az_K\to [0,\infty)$ is upper semi-continuous). 
\end{lemma}
\begin{proof}

The proof is essentially the same as \cite[Lemma~3.5]{LR:00}, so we give only a brief sketch:
On each basic open set, we see that the adele norm is the infimum of continuous functions $||\cdot||_F\colon \Az_K\to [0,\infty)$ defined by $||a||_F:=\prod_{v\in F}|a|_v$ for $a \in \Az_K$ and large finite subsets $F \subseteq \Sigma_K$. 
\end{proof}

The next lemma is the key technical observation in this section; it is a generalisation of an argument in the proof of \cite[Theorem~3.7]{LR:00}.

\begin{lemma}
\label{lem:oinCL}
Let $p$ be a rational prime. Suppose we have a sequence of ideles $a_k\in\Az_K^*$ and a strictly increasing sequence of positive integers $n_k$ such that $||a_k||^{1/d}\in (p^{-(n_k+1)},p^{-n_k}]$ for all $k\in\Zz_{>0}$, where $d:=[K:\Qz]$.
Then, there exists a sequence of $x_k\in K^*$ such that the sequence $b_k:=p^{n_k+1}x_ka_k$ has a convergent subsequence (in the topology of $\Az_K$) with limit $b\in \Az_K$ satisfying $\cZ(b)=\{\p\in\cP_K : p\in \p\}$.
\end{lemma}
\begin{proof}
First, we need a general observation, which is inspired by the proofs on \cite[p.143]{Lang}. By \cite[p.66]{CF}, there exists a constant $c\geq 1$, depending only on $K$, such that for every $a\in\Az_K^*$ with $||a||>c$, there exists $x\in K^*$ with $1\leq |xa|_v$ for all $v\in\Sigma_K$; this implies that we also have
\[
|xa|_v=\frac{||xa||}{\prod_{w\neq v}|xa|_w}\leq ||xa||=||a||\quad (v\in\Sigma_K).
\]
For each $k\in\Zz_{>0}$, define $t_k\in\Az_K^*$ by $(t_k)_\infty:=(||a_k||^{1/d},...,||a_k||^{1/d})\in\prod_{v\in\Sigma_{K,\infty}}K_v^*$ and $(t_k)_f:=(1,1,1,...)\in\Az_{K,f}^*$. Then, $||t_k||=||a_k||$. Put $a_k':=t_k^{-1}a_k$, so that $a_k=t_ka_k'$ and $||a_k'||=1$. Choose $r>c$ and define $u\in\Az_K^*$ by $u_\infty:=(r^{1/d},...,r^{1/d})\in \prod_{v\in\Sigma_{K,\infty}}K_v^*$ and $u_f:=(1,1,1,...)\in\Az_{K,f}^*$. Then, for every $k\in\Zz_{>0}$, we have that $ua_k'\in\Az_K^*$ satisfies $||ua_k'||=||u||=r>c$, so that there exists $x_k\in K^*$ with $1\leq |x_kua_k'|_v\leq r$ for all $v\in\Sigma_K$. For $\p\in\cP_K$, $|x_kua_k'|_\p$ is of the form $N(\p)^l$ for $l\in\Zz_{\geq 0}$, so for all $\p$ with $N(\p)>r$, we must have $|x_kua_k'|_\p=1$. Hence, there exists a finite set $S\subseteq \Sigma_K$ with $\Sigma_{K,\infty}\subseteq S$ such that 
\begin{equation}
    \label{eqn:prodannulus}
x_kua_k'\in \prod_{v\in S}\{x\in K_v^* : 1\leq |x|_v\leq r\}\times\prod_{\p\notin S}\O_\p^*
\end{equation}
for all $k$.
Since $(t_kp^{n_k+1})_\infty\in (1,p]^{\Sigma_{K,\infty}}$ and $p^{n_k+1}\in\O_\p^*$ for all primes $\p$ with $p\notin\p$, we see that $t_kx_kp^{n_k+1}a_k'=x_kp^{n_k+1}a_k$ lies in the set
\[
u^{-1}\left(\prod_{v\in S\setminus S_p}\{x\in K_v : 1\leq |x|_v\leq p^2r\}\times \prod_{\p\in S_p}\O_\p\times\prod_{\p\notin S\cup S_p}\O_\p^*\right)
\]
for all $k$ such that $\max_{\p\in S_p}|p^{n_k+1}|_\p < 1/r$, where $S_p:=\{\p\in\cP_K : p\in \p\}$. Since this set is compact in $\Az_K$, $x_kp^{n_k+1}a_k$ has a convergent subsequence in the topology of $\Az_K$, converging to some $b\in\Az_K$. Moreover, we have $\cZ(b)=S_p$ because $r^{-2/d}\leq |x_kp^{n_k+1}a_k|_v$ for all $v\not\in S_p$ and $|x_kp^{n_k+1}a_k|_\p\leq |x_ka_k|_\p N(\p)^{-(n_k+1)}\leq N(\p)^{-(n_k+1)}r$ for all $\p\in S_p$. Here, we used that $|x_kua_k'|_v\leq r$ for every $v\in \Sigma_K$ by \eqref{eqn:prodannulus} and $|x_kua_k'|_\p=|x_ka_k|_\p$ for all $\p\in\cP_K$ since $ua_k'$ and $a_k$ have the same finite part: $(ua_k')_f=(a_k)_f$.
\end{proof}

The following proposition completes our description of the topology on $2^{\Sigma_K}\sqcup C_K$. 

\begin{proposition} \label{prop:qick}
 For $B\subseteq C_K$, we have 
\[
\ol{B}^{\rm qo}=
\begin{cases}
 \ol{B}^{\id}   & \text{ if } 0\notin \overline{\{||b|| : b\in B\}},\\
2^{\Sigma_K}\sqcup C_K & \text{ if } 0\in \overline{\{||b|| : b\in B\}}.
\end{cases}
\]
\end{proposition}

Note that Proposition~\ref{prop:qick} implies that the relative topology of $C_K$ with respect to the topology of $2^{\Sigma_K} \sqcup C_K$ is different from the idelic topology.

\begin{proof}
Let $\eps>0$, let $q \colon \Az_K^* \to C_K$ be the quotient map, and put $C_K^\eps = q(\Az_K^\eps)$. 
Then, the quasi-orbit map $\Xi \colon (\Az_K,\, {\rm ad}) \to (2^{\Sigma_K} \sqcup C_K,\, {\rm qo})$ restricts to a surjective map $\xi \colon (\Az_K^\eps,\ {\rm ad}) \to (C_K^\eps,\,  {\rm qo})$. By Lemma~\ref{lem:coincide}, $\xi$ is continuous with respect to the idelic topology for $\Az_K^\eps$. In addition, since $\Az_K^\eps$ is a closed $K^*$-invariant subset of $\Az_K$ by Lemma~\ref{lem:usc}, $\xi$ is an open map. Hence, $\xi$ induces an open continuous map $\ol{\xi} \colon (C_K^\eps,\, {\rm id}) \to (C_K^\eps,\, {\rm qo})$ such that the following diagram commutes:
\[  \begin{tikzcd}
  \Az_K^\eps \arrow[d,"q"] \arrow[rr,"\xi"] & &  C_K^\eps \\
  C_K^\eps \arrow[rru,"\ol{\xi}"']. & &
  \end{tikzcd}\]
Since $\xi$ coincides with the restriction of $q$ by definition, $\ol{\xi}$ is the identity map. Hence, the quasi-orbit topology and the idelic topology for $C_K^\eps$ coincide. In particular, for $B \subseteq C_K$ with $0\notin \overline{\{||b|| : b\in B\}}$, there exists $\eps>0$ with $B \subseteq C_K^\eps$ so that we have $\ol{B}^{\id}=\ol{B}^{\rm qo}$.  

Now assume $0\in \overline{\{||b|| : b\in B\}}$. Let $F\subseteq\Sigma_K$ be any finite subset. Choose a rational prime $p$ such that $F\cap\{\p\in\cP_K : p\in \p\}=\emptyset$.
Then, there exists a strictly increasing sequence of natural numbers $n_1<n_2<\cdots$ and a sequence of ideles $a_k$ with $q(a_k)\in B$ such that $||a_k||^{1/d}\in (p^{-(n_k+1)},p^{-n_k}]$ for all $k\in\Zz_{>0}$. By Lemma~\ref{lem:oinCL}, there exists $x_k\in K^*$ such that $b_k:=p^{n_k+1}x_ka_k$ has a convergent subsequence in the adelic topology with limit $b\in \Az_K$ satisfying $\cZ(b)=\{\p\in\cP_K : p\in \p\}$. Since $\Xi$ is continuous, $\Xi(b)=\cZ(b)$ lies in $\ol{B}^{\rm qo}$. Since $\cZ(b)\cap F=\emptyset$, $\cZ(b)\in U_F$. 
Hence, $\ol{B}^{\rm qo}\cap U_F\neq\emptyset$ for every finite subset $F\subseteq \Sigma_K$. By Proposition~\ref{prop:qipc}, $2^{\Sigma_K}\cap \ol{B}^{\rm qo}$ is closed in the power-cofinite topology of $2^{\Sigma_K}$; since $2^{\Sigma_K}\cap \ol{B}^{\rm qo}$ intersects every basic open set of $2^{\Sigma_K}$ non-trivially, we have $2^{\Sigma_K}=2^{\Sigma_K}\cap \ol{B}^{\rm qo}$. Hence, $\ol{B}^{\rm qo}=2^{\Sigma_K}\sqcup C_K$ by Proposition~\ref{prop:qipc}.
\end{proof}

The topology of $2^{\Sigma_K} \sqcup C_K$ is completely determined by Proposition~\ref{prop:qipc} and Proposition~\ref{prop:qick}. However, it will be convenient to have an explicit basis of open sets for $2^{\Sigma_K} \sqcup C_K$:

\begin{theorem}
\label{thm:basicopensets}
 Let $F$ be a finite subset of $\Sigma_K$, let $\varepsilon>0$, and let $V \subset C_K$ be an open set (with respect to the usual idelic topology of $C_K)$ such that $||a||>\varepsilon$ for every $a \in V$. Then, the set $U(F,\varepsilon,V):= U_F \sqcup B_\varepsilon \sqcup V$ is an open set in $2^{\Sigma_K} \sqcup C_K$, where $B_\varepsilon = \{a \in C_K \colon ||a||<\varepsilon\}$. Moreover, $\{U(F,\varepsilon,V) \colon F,\varepsilon,V\}$ is a basis of open subsets for the quasi-orbit topology of $2^{\Sigma_K} \sqcup C_K$.

\end{theorem}
\begin{proof}
The set $U(F,\varepsilon,V)$ is open: We have
 \[ (U_F \sqcup B_\varepsilon \sqcup V)^c = (2^{\Sigma_K}\setminus U_F) \sqcup (C_K\setminus (B_\varepsilon \sqcup V)), \]
the set $2^{\Sigma_K}\setminus U_F$ is closed by Proposition~\ref{prop:qipc}, and the set $C_K\setminus (B_\varepsilon \sqcup V)$ is closed by Proposition \ref{prop:qick}.

Let $x \in 2^{\Sigma_K} \sqcup C_K$, and let $A \subseteq 2^{\Sigma_K} \sqcup C_K$ be a closed set that does not contain $x$. Let  
 \[ A_1 = \ol{A \cap 2^{\Sigma_K}}^{\rm qo}, \mbox{ and } A_2 = \ol{A \cap C_K}^{\rm qo}. \]
 Then, we have $A = A_1 \cup A_2$. Since $A \neq 2^{\Sigma_K} \sqcup C_K$, $A_1$ is closed in $2^{\Sigma_K}$ with respect to the power-cofinite topology by Proposition~\ref{prop:qipc}, and $A_2$ is closed in $C_K$ with respect to the idelic topology by Proposition~\ref{prop:qick}. By Proposition~\ref{prop:qick}, we can choose $\varepsilon > 0$ sufficiently small so that $B_\varepsilon \cap A_2 = \emptyset$. 
Suppose $x \in 2^{\Sigma_K}$. Then, there exists a finite set $F \subseteq \Sigma_K$ such that $x \in U_F$ and $U_F \cap A_1 = \emptyset$. Then, we have $x \in U(F,\varepsilon,\emptyset)$ and $U(F,\varepsilon,\emptyset) \cap A = \emptyset$. Now suppose $x \in C_K$. We may assume that $\eps<||x||$. Then, there exists an open set $V \subseteq C_K$ with respect to the idelic topology of $C_K$ such that $x \in V$ and $V \cap ( B_\varepsilon \cup A_2) = \emptyset$. Take $F \subseteq \Sigma_K$ with $U_F \cap A_1 = \emptyset$, so that we have $x \in U(F,\varepsilon,V)$ and $U(F,\varepsilon,V) \cap A = \emptyset$.
 Therefore, $\{U(F,\varepsilon,V) \colon F,\varepsilon,V\}$ is a basis. 
\end{proof}

We now turn to the primitive ideal space. Recall that for a topological space $X$ and $x,y \in X$, the point $y$ is said to be a generalisation of $x$ if $x$ belongs to the closure of $\{y\}$ (see \cite[p.93]{Hartshorne}). Williams's theorem \cite[Therem~8.39]{Will:book} provides a canonical continuous, open, surjective map $\varphi\colon \Prim(\A_K)\to 2^{\Sigma_K}\sqcup C_K$ (cf. \cite[Lemma~4.6]{BT}). Using this and our description of the quasi-orbit space $2^{\Sigma_K}\sqcup C_K$, we obtain the following explicit description of $\Prim(\A_K)$:

\begin{proposition} \label{prop:primA}
    We have a set-theoretic decomposition 
    \[ \Prim(\A_K) = \widehat{K^*} \sqcup (2^{\Sigma_K} \setminus \{\Sigma_K\}) \sqcup C_K.\]
    The set of closed points of $\Prim(\A_K)$ is equal to $\widehat{K^*} \sqcup C_K$. For a closed point $P$ of $\Prim(\A_K)$, $P \in C_K$ if and only if 
    no point in $\Prim(\A_K)$ other than $P$ and $\emptyset \in 2^{\Sigma_K}$ is a generalisation of $P$. 
\end{proposition}
Note that $\emptyset \in 2^{\Sigma_K}$ is characterised as the unique dense point in $2^{\Sigma_K} \sqcup C_K$. The topology of $\Prim(\A_K)$ is completely characterized by Theorem~\ref{thm:basicopensets}, Proposition~\ref{prop:primA}, and the fact that the canonical map $\varphi\colon \Prim (\A_K) \to 2^{\Sigma_K} \sqcup C_K$ is open and continuous.
\begin{proof}
By Proposition~\ref{prop:qipc} and Proposition~\ref{prop:qick}, the set of closed points of $2^{\Sigma_K} \sqcup C_K$ is equal to $\{\Sigma_K\} \sqcup C_K$. Hence, the same argument as the proof of \cite[Proposition~4.13]{BT} based on \cite[Theorem~8.39]{{Will:book}} yields the decomposition $\Prim(\A_K) = \widehat{K^*} \sqcup (2^{\Sigma_K} \setminus \{\Sigma_K\}) \sqcup C_K$ and the claim that the set of closed points is exactly equal to $\widehat{K^*} \sqcup C_K$.

Let $P \in \widehat{K^*} \subseteq \Prim (\A_K)$.
Then, for any open neighbourhood $U \subseteq \Prim (\A_K)$ of $P$, we have $2^{\Sigma_K} \subseteq \varphi(U)$ by Theorem~\ref{thm:basicopensets}, since the basic open set $B(F,\varepsilon,V)$ from Theorem~\ref{thm:basicopensets} contains $\Sigma_K$ if and only if $F=\emptyset$. Hence, all points in $2^{\Sigma_K} \setminus \{\Sigma_K\} \subseteq \Prim (\A_K)$ are generalisations of $P$. Conversely, let $P \in C_K \subseteq \Prim (\A_K)$ and let $x=\varphi(P) \in C_K$. If $Q \in \Prim (\A_K)$ is a generalisation of $P$ and $y=\varphi(Q)$, then $x \in \ol{\{y\}}$ by continuity of $\varphi$, which implies that $y=x \in C_K$ or $y=\emptyset \in 2^{\Sigma_K}$ by Proposition~\ref{prop:qipc} and Proposition~\ref{prop:qick}. Since $\varphi^{-1}(\varphi(Q))=\{Q\}$, we have $Q=P$ or $Q=\emptyset$. 
\end{proof}

\begin{remark}
\label{rmk:charprim}
Similarly to \cite[Remark~3.9]{LR:00}, we also have the following characterisation of primitive ideals.
If $\pi$ is an irreducible representation of $\A_K$, then
\begin{enumerate}[\upshape(i)]
\item $\ker\pi$ comes from a point in $\widehat{K^*}$ if and only if $\pi$ is one-dimensional;
\item $\ker\pi$ comes from a point in $C_K$ if and only if $\pi$ is CCR, that is, $\im\:\pi=\cK(H_\pi)$.
\end{enumerate}
\end{remark}

We briefly summarise the necessary facts on the primitive ideal space of $\B_K$. Since the quasi-orbit space of $\Gamma_K \acts \Az_K/\mu_K$ is also identified with $2^{\Sigma_K} \sqcup C_K$, the C*-algebra $\B_K$ is also a C*-algebra over $2^{\Sigma_K} \sqcup C_K$. 
\begin{definition}[cf.~{\cite[Definition~4.16]{BT}}]
\label{def:varphiK}
Let $\varphi_K \colon \B_K \to \A_K$ be the *-homomorphism obtained by applying \cite[Proposition~3.1]{BT} to $G=K^*$, $\mu=\mu_K$, $X=\Az_K$, and $\chi=1$ (the trivial character). 
\end{definition}
\begin{remark}
\label{rmk:Bk}
The statements in Proposition~\ref{prop:primA} also hold for $\B_K$ with $K^*$ replaced by $\Gamma_K$ using exactly the same arguments.
Additionally, the same argument as the proof of \cite[Proposition~4.17]{BT} implies that $\varphi_K$ is $(2^{\Sigma_K}\sqcup C_K)$-equivariant.
 \end{remark}

\section{Subquotients from places} 
\label{sec:subquotients}

In this section, we make several preparations on subquotients of $\A_K$. Define 
\[
\Upsilon\colon 2^{\Sigma_K}\sqcup C_K \to 2^{\Sigma_K},\quad \begin{cases}
 \Upsilon(T)=T &\text{ for } T\in 2^{\Sigma_K},\\
 \Upsilon(a)=\emptyset &\text{ for }a\in C_K.
 \end{cases}
\]

\begin{lemma} \label{lem:ctsopen}
 The map $\Upsilon$ is $({\rm qo},{\rm pc})$-continuous and open.
\end{lemma}
\begin{proof}
For each finite subset $F\subseteq\Sigma_K$, $\Upsilon^{-1}(U_F)=U_F\sqcup C_K$ is open in $2^{\Sigma_K}\sqcup C_K$ by Theorem~\ref{thm:basicopensets}. Thus, $\Upsilon$ is continuous. In addition, for a basic open set $U(F,\varepsilon,V) \subseteq 2^{\Sigma_K}\sqcup C_K$ from Theorem~\ref{thm:basicopensets}, we have $\Upsilon(U(F,\varepsilon,V)) = U_F$, which implies that $\Upsilon$ is open.
\end{proof}

Let $\psi_K$ and $\eta_K$ be the compositions of $\Upsilon$ with the canonical surjections $\Prim(\A_K)\to 2^{\Sigma_K}\sqcup C_K$ and $\Prim(\B_K)\to 2^{\Sigma_K}\sqcup C_K$, respectively. Then, $\psi_K$ and $\eta_K$ are continuous, open surjections by Lemma~\ref{lem:ctsopen} and \cite[Lemma~4.6]{BT}, so that $(\A_K,\psi_K)$ and $(\B_K,\eta_K)$ are C*-algebras over $2^{\Sigma_K}$.
Note that the map $\varphi_K\colon\B_K\to\A_K$ from Definition~\ref{def:varphiK} is $2^{\Sigma_K}$-equivariant.

\begin{lemma}[cf. {{\cite[Lemma~4.15]{BT}}~and~\cite[Lemma~2.10]{KT}}]
\label{lem:over2Sigma}
 Let $K$ and $L$ be number fields, and suppose $\alpha\colon \A_K\to \A_L$ is a *-isomorphism. Then, there exists a bijection $\theta\colon \Sigma_K\to \Sigma_L$ such that the following diagram commutes:
\[  \begin{tikzcd}
  \Prim(\A_K)\arrow[d,"\psi_K"]\arrow[rr,"P\mapsto \alpha(P)"] & & \Prim(\A_L)\arrow[d,"\psi_L"]\\
  2^{\Sigma_K} \arrow[rr,"\tilde{\theta}"] & & 2^{\Sigma_L}\nospacepunct{,}
  \end{tikzcd}\]
  where $\tilde{\theta}$ denotes the homeomorphism $2^{\Sigma_K}\simeq 2^{\Sigma_L}$ induced by $\theta$. Furthermore, the same statement holds for $\A_K$ and $\A_L$ replaced by $\B_K$ and $\B_L$, respectively.
\end{lemma}
\begin{proof}
By Proposition~\ref{prop:primA}, the topological space $2^{\Sigma_K}$ is obtained by identifying all closed points $x \in \Prim(\A_K)$ which admit a generalisation other than $x$ and the unique dense point $\emptyset$, and identifying $\emptyset$ and all closed points $x \in \Prim(\A_K)$ which do not admit a generalisation other than $x$ and $\emptyset$. The rest of the proof is the same as that of \cite[Lemma~4.15]{BT}, and the proof for $\B_K$ and $\B_L$ is exactly the same (see Remark~\ref{rmk:Bk}).
\end{proof}

\begin{remark}
Let us explain the explicit relationship between $\A_K$ and $\B_K$ and the C*-algebras associated with the finite adele ring from \cite{BT}. 
Let $P_{K,\infty}$ denote the primitive ideal of $\A_K$
identified with $\Sigma_{K,\infty} \in 2^{\Sigma_K} \subseteq \Prim(\A_K)$. 
Explicitly, $P_{K,\infty}=C_0(\Az_K \setminus \Az_{K,f}) \rtimes K^*$ by \cite[Proposition~4.7]{BT}, 
so that $\A_K/P_{K,\infty}=C_0(\Az_{K,f})\rtimes K^*=\fA_K$. Similarly, if we let $Q_{K,\infty}$ be the primitive ideal of $\B_K$
identified with $\Sigma_{K,\infty} \in 2^{\Sigma_K} \subseteq \Prim(\B_K)$, 
then $\B_K/Q_{K, \infty}=C_0(\Az_{K,f}/\mu_K)\rtimes \Gamma_K=\fB_K$.
\end{remark}

As explained in a more general setting in \cite[\S~4.2]{BT}, every finite subset $F\subseteq \Sigma_K$ gives rise to a subquotient $\A_K(\{F^c\})$, and for every $v\in F^c$, there is an extension $\cE_{\A_K}^{F,v}$ (see \cite[Definition~4.1]{BT}).

\begin{definition}
We let $\A_K^\emptyset:=\A_K(\{\Sigma_K\})$ and $\B_K^\emptyset:=\B_K(\{\Sigma_K\})$. Given $v\in\Sigma_K$, put $\A_K^v:=\A_K(\{\{v\}^c\})$ and $\B_K^v:=\B_K(\{\{v\}^c\})$. We also let $\cF_K^v:= \cE_{\A_K}^{\emptyset, v}$ and $\cE_K^v:= \cE_{\B_K}^{\emptyset, v}$ be the extensions associated with $\emptyset$ and $v \in \Sigma_K$.
Denote by $\partial_K^v$ the boundary map associated with the six-term exact sequence arising from $\cE_K^v$.
\end{definition}

We let $\A_K^{v,c}$ and $\B_K^{v,c}$ denote the middle terms of the extensions $\cF_K^v$, and $\cE_K^v$, respectively. That is, $\cF_K^v$, and $\cE_K^v$ are given by
\[
\cF_K^v\colon 0\to\A_K^v\to \A_K^{v,c}\to \A_K^\emptyset\to 0
\]
and
\[
\cE_K^v\colon 0\to\B_K^v\to \B_K^{v,c}\to \B_K^\emptyset\to 0.
\]

The general results from \cite[\S~4]{BT} give the following:
\begin{proposition}
We have a canonical isomorphisms
\[
\A_K^\emptyset = C^*(K^*),\quad  \B_K^\emptyset =C^*(\Gamma_K),\quad \A_K^v = C_0(K_v^*) \rtimes K^*, \quad\text{and}\quad \B_K^v = C_0(K_v^*/\mu_K) \rtimes \Gamma_K,
\]
where $v\in\Sigma_K$.
\end{proposition}
\begin{proof}
Each isomorphism is obtained from \cite[Proposition~4.8]{BT}. For the first, take, in the notation from \cite[Proposition~4.8]{BT}, $X=\Az_K$, $G=K^*$, and $Z=\{\Sigma_K\}$. The other three are obtained similarly.
\end{proof}

\begin{remark} \label{rmk:description}
For $v\in\Sigma_{K,\infty}$, the extension $\cE_K^v$ can be described explicitly as follows:
\begin{equation}
\label{eqn:extinfinite}
\cE_K^v\colon \begin{cases}
0 \to C_0(\Rz^*/\mu_K) \rtimes \Gamma_K \to C_0(\Rz/\mu_K) \rtimes \Gamma_K \to C^*(\Gamma_K) \to 0  & \text{ if } v \text{ real},\\
0 \to C_0(\Cz^*/\mu_K) \rtimes \Gamma_K \to C_0(\Cz/\mu_K) \rtimes \Gamma_K \to C^*(\Gamma_K) \to 0  & \text{ if } v \text{ complex},
\end{cases}
\end{equation}
where all maps are the canonical ones. The extension $\cF_K^v$ can be described similarly.
When $v$ is finite, the projection maps $\A_K \to \fA_K$ and $\B_K \to \fB_K$ restrict to identifications of $\cF_K^v$ and $\cE_K^v$ with the extensions from \cite[Definition~4.22]{BT}.
In particular, an explicit description of the extensions $\cF_K^v$ and $\cE_K^v$ follows from \cite[\S~4.5]{BT}.
\end{remark}

\section{K-theoretic distinction of real, complex, and finite places}
\label{sec:main}
We now establish several general results on \K-theory and boundary maps for certain actions of infinitely generated free abelian groups as preparation for the proof of our main theorem (Theorem~\ref{thm:main}). We first observe a consequence of the Baum--Connes conjecture:
\begin{proposition} \label{prop:KKhomotopy}
Let $A$ be a separable C*-algebra, and let $\Gamma$ be a free abelian group. Let $\gamma^0, \gamma^1 \colon \Gamma \acts A$ be two actions. Suppose that $\gamma^0$ and $\gamma^1$ are homotopic in the sense that 
there exists an action $\ol{\gamma} \colon \Gamma \acts C([0,1],A)$ such that $(\ol{\gamma}_sf)(t)=\gamma_s^t(f(t))$ for $t \in \{0,1\}$ and $s \in \Gamma$. 
For $t \in \{0,1\}$, let $\eval_t \colon C([0,1],A) \rtimes_{\ol{\gamma}} \Gamma \to A \rtimes_{\gamma^t} \Gamma$ be the map induced by evaluation at $t$. 
Then, $\eval_0$ and $\eval_1$ induce \KK-equivalences. 
\end{proposition}
\begin{proof}
For a *-homomorphism $\alpha$ between C*-algebras, we let $[\alpha]_\KK$ denote the associated element in \KK-theory. We have the short exact sequence 
\[ 0 \to C_0([0,1),A) \rtimes_{\ol{\gamma}} \Gamma \to  C([0,1],A) \rtimes_{\ol{\gamma}} \Gamma \xrightarrow{\eval_1} A \rtimes_{\gamma} \Gamma \to 0. 
\]
Since $C_0((0,1],A) \rtimes_{\ol{\gamma}} \Gamma$ is \KK-equivalent to $0$ by \cite[Theorem~9.3]{MeyerNest:06}, the maps
$(- \ho_A [\eval_1]_{\KK}) \colon \KK(D,A) \to \KK(D,B)$ and $([\eval_1]_{\KK} \ho_B -) \colon \KK(B,D) \to \KK(A,D)$ are isomorphisms for any separable C*-algebra $D$ by the six-term exact sequence in \KK-theory \cite[Theorem 19.5.7]{B}. Hence, $[\eval_1]_{\KK}$ is a \KK-equivalence, and we can see that $[\eval_0]_{\KK}$ is also a \KK-equivalence in the same way. 
\end{proof}

Let $\mu \subseteq \Tz$ be a finite subgroup. Define extensions $\cR$, $\cR_1$, and $\cC_\mu$ by 
\begin{equation}\begin{tikzcd}[
  row sep = 0ex,
  column sep = 1ex,
/tikz/column 1/.append style={anchor=base east},
 /tikz/column 2/.append style={anchor=base west},
  ampersand replacement=\&,
]
 \cR \colon \& 0  \to C_0(0,\infty) \to C_0[0,\infty) \to \Cz \to 0, \\
  \cR_1  \colon \& 0  \to C_0(\Rz^*/\pmone) \to C_0(\Rz/\pmone) \to \Cz \to 0, \\
   \cC_\mu \colon  \& 0  \to C_0(\Cz^*/\mu) \to C_0(\Cz/\mu) \to \Cz \to 0,
\end{tikzcd}\end{equation}
where all of the rightmost *-homomorphisms are evaluations at $0$. The boundary map of $\cR$ is the Bott map, which is an isomorphism (of degree 1). Let $\beta \in \K_1(C_0(0,\infty))$ denote the generator satisfying  $\partial_\cR([1_\Cz]_0) = \beta$.

Throughout this section, $\cX$ denotes $\cR_1$ or $\cC_\mu$ for some $\mu \subseteq \Tz$. If $\cX=\cR_1$, then we put $X:=\Rz/\pmone$ and $X^*:=\Rz^*/\pmone$. Similarly, if $\cX=\cC_\mu$, then we let $X:=\Cz/\mu$ and $X^*:=\Cz^*/\mu$. 
Then, the extension $\cX$ is described as 
\[ \cX \colon 0 \to C_0(X^*) \to C_0(X) \to \Cz \to 0.\]

Let $\Gamma$ be a free abelian subgroup of the multiplicative group $X^*$. Let $\gamma$ be the action of $\Gamma$ on $X$ by multiplication. We also denote by $\gamma$ the restriction of $\gamma$ to $X^*$. Let $\cX \rtimes_\gamma \Gamma$ denote the extension
\[ 
\cX \rtimes_\gamma \Gamma \colon 0 \to C_0(X^*) \rtimes_{\gamma} \Gamma \to C_0(X) \rtimes_{\gamma} \Gamma \to C^*(\Gamma) \to 0. 
\]
Let $\Abs \colon X \to [0,\infty)$ be the continuous and  proper map $z \to |z|$. Then, the restriction of $\Abs$ to $X^*$ has a splitting (note that if $\cX=\cR_1$, then $\Abs$ is a homeomorphism), so that $(0,\infty)$ is a direct product factor of $X^*$. We let $A:=\Cz$ if $\cX=\cR_1$ and $A:=C(\Tz/\mu)$ if $\cX=\cC_\mu$, and we make the identification $C_0(X^*)= C_0(0,\infty) \otimes A$ by way of $\Abs$.

\begin{lemma} \label{lem:Bott}
Let $\iota \colon \Cz \to A$ be the inclusion map. 
Then, the following diagram commutes: 
\[  \begin{tikzcd}
  \K_*(\Cz) \arrow[rr,"\iota_*"] \arrow[drr,"\partial_\cX"] & &  \K_*(A) \arrow[d, "\partial_{\cR \otimes A}"] \\
  & & \K_{*+1}(C_0(X^*)). 
  \end{tikzcd}\]
\end{lemma}
\begin{proof}
We have the following commutative diagram with exact rows: 
 \begin{equation} \begin{tikzcd}
 	\cR \colon &
	0 \arrow[r] &  C_0(0,\infty) \arrow[d, "\Abs^*"]  \arrow[r] &  C_0[0,\infty)  \arrow[d, "\Abs^*"] \arrow[r] &  \Cz \arrow[d, "\id"]  \arrow[r] &  0 \\ 
	\cX \colon & 0 \arrow[r] & 
	C_0(X^*) \arrow[r] & C_0(X)  \arrow[r]  &\Cz \arrow[r] &  0\nospacepunct{.}
 \end{tikzcd}\end{equation}
 Hence, $\partial_{\cX} = (\Abs^*)_* \circ \partial_{\cR}$ by naturality of the boundary map. 
 Since $\partial_{\cR \otimes A}=\partial_\cR \otimes \id$ is invertible, the following calculation completes the proof:
  \[ (\partial_{\cR \otimes A}^{-1} \circ \partial_\cX)([1_\Cz]_0) = (\partial_{\cR \otimes A}^{-1} \circ  (\Abs^*)_* \circ \partial_{\cR})([1_\Cz]_0) = \partial_{\cR \otimes A}^{-1}(\beta \otimes [1_A]_0) = \partial_\cR^{-1}(\beta) \otimes [1_A]_0 =[1_A]_0.\qedhere 
  \]
\end{proof}

The following result is the key observation on boundary maps.
\begin{proposition} \label{prop:bijinj}
The boundary map $\K_*(C^*(\Gamma)) \to \K_{*+1}(C_0(\Rz^*/\pmone) \rtimes_{\gamma} \Gamma)$  of $\cR_1 \rtimes_\gamma \Gamma$ is bijective. In addition, for any finite subgroup $\mu \subseteq \Tz$, the boundary map $\K_*(C^*(\Gamma)) \to \K_{*+1}(C_0(\Cz^*/\mu) \rtimes_{\gamma} \Gamma)$ of $\cC_\mu \rtimes_\gamma \Gamma$ is injective but not surjective. 
\end{proposition}
\begin{proof}
First, we show that $\gamma$ is homotopic to the trivial action. Fix a basis $B$ of $\Gamma$. 
For each $b \in B$, let $x_b \colon [0,1] \to X^*$ be a path in $X^*$ satisfying $x_b(0)=1$ and $x_b(1)=b$. Such a path exists because $X^*$ is path-connected. For each $t \in [0,1]$, let $\gamma_b^t \colon \Zz \acts  C_0(X)$ be the action given by multiplication with $x_b(t)$.  
Since $\{\gamma_b^t \colon b\in B, t \in [0,1]\}$ is a commuting family of *-automorphisms, we can 
define an action $\ol{\gamma} \colon \Gamma \acts C([0,1], C_0(X))$ by $(\ol{\gamma}_b(f))(t)= \gamma_b^t(f(t))$ for $b \in B$, $f \in C([0,1], C_0(X))$, and $t \in [0,1]$. Then, $\ol{\gamma}$ is a homotopy between $\gamma$ and the trivial action. 

We have the following commutative diagram with exact rows: 
 \begin{equation} \begin{tikzcd}[column sep=2ex]
	\cX \otimes C^*(\Gamma) \colon & 0\arrow{r}&	
	C_0(X^*) \otimes C^*(\Gamma) \arrow[r]  & C_0(X) \otimes C^*(\Gamma) \arrow[r] & C^*(\Gamma) \arrow[r] & 0\\
	  & 0\arrow{r}&	
	 C([0,1], C_0(X^*)) \rtimes_{\ol{\gamma}} \Gamma \arrow[u, "\eval_0"] \arrow[d, "\eval_1"] \arrow[r] &
	C([0,1], C_0(X)) \rtimes_{\ol{\gamma}} \Gamma \arrow[u, "\eval_0"] \arrow[d, "\eval_1"] \arrow[r] & 
	C[0,1] \otimes C^*(\Gamma) \arrow[u, "\eval_0"] \arrow[d, "\eval_1"] \arrow[r] & 0\\
	\cX \rtimes_\gamma \Gamma \colon & 
	0\arrow{r}&	C_0(X^*) \rtimes_{\gamma} \Gamma \arrow[r] & C_0(X) \rtimes_{\gamma} \Gamma \arrow[r] & C^*(\Gamma) \arrow[r] & 0\nospacepunct{.} 
 \end{tikzcd}\end{equation}
 By Proposition \ref{prop:KKhomotopy}, all vertical maps in the first two columns induce KK-equivalences. For the right column, $\eval_t \colon C[0,1] \otimes C^*(\Gamma) \to C^*(\Gamma)$ induces the same map in K-theory for $t \in \{0,1\}$, since $\eval_0$ and $\eval_1$ are homotopic. Hence, it suffices to show that $\partial_{\cX \otimes C^*(\Gamma) }$ has the desired property. 
 Under the identification $\K_0(\Cz) \otimes \K_*(C^*(\Gamma)) = \K_*(C^*(\Gamma))$, we have $\partial_{\cX \otimes C^*(\Gamma)} = \partial_{\cX} \otimes \id_{\K_*(C^*(\Gamma))}$ (see \cite[\S~2.2]{BT}). Hence, by Lemma~\ref{lem:Bott}, $\partial_{\cX \otimes C^*(\Gamma) }$ has the desired property if and only if $\iota_* \colon \K_*(\Cz) \to \K_*(A)$ does. If $\cX=\cR_1$, then $\iota_*$ is the identity map (note that $A=\Cz$ in this case), which is clearly bijective. Suppose $\cX=\cC_\mu$. Then, since $\Tz/\mu$ is homeomorphic to $\Tz$, we have $\K_*(A) = \Zz[1]_0 \oplus \Zz a$, where $a \in \K_1(C(\Tz/\mu))$ is a generator (note that $A=C(\Tz/\mu)$ in this case). Hence, $\iota_*$ is injective but not surjective. 
\end{proof}

We are now ready for our main result on distinguishing places.
\begin{theorem} \label{thm:main}
    The boundary map $\partial_K^v \colon \K_*(\B_K^\emptyset) \to \K_*(\B_K^v)$ is bijective if and only if $v$ is real, is injective but not surjective if and only if $v$ is complex, and is not injective if and only if $v$ is finite. 
\end{theorem}

\begin{proof}
The case that $v$ is finite follows from \cite[Lemma~5.10]{BT}, and the case that $v$ is infinite follows from Remark~\ref{rmk:description} combined with Proposition~\ref{prop:bijinj}. 
\end{proof}

The boundary maps $\partial_K^v$ in Theorem~\ref{thm:main} are invariants of $\B_K$ as a C*-algebra over the power set of a countable infinite set, see \cite[\S~4.2]{BT}.
For number fields $K$ and $L$, any *-isomorphism between $\B_K$ and $\B_L$ is automatically $2^{\Sigma_K}$-equivariant, after identifying $\Sigma_K$ and $\Sigma_L$ by Lemma~\ref{lem:over2Sigma}, so real, complex, and finite places are distinguishable from $\B_K$. For $\A_K$, we have the following consequence:

\begin{corollary} 
\label{cor:main}
Let $\alpha \colon \A_K \to \A_L$ be a  *-isomorphism, and let $\theta \colon \Sigma_K \to \Sigma_L$ be the bijection induced by $\alpha$ from Lemma~\ref{lem:over2Sigma}. 
Then, $v$ is real if and only if $\theta(v)$ is real, $v$ is complex if and only if $\theta(v)$ is complex, and $v$ is finite if and only if $\theta(v)$ is finite. As a consequence, $\A_K$ is *-isomorphic to $\A_L$ if and only if $K$ is isomorphic to $L$. 
\end{corollary}
\begin{proof}
The *-isomorphism $\alpha$ induces an isomorphism of extensions $\cF_K^v \to \cF_L^{\theta(v)}$, that is, a triplet of *-isomorphisms which makes the following diagram commute: 
 \begin{equation} \begin{tikzcd}
	\cF_K^v \colon 0\arrow{r}&	\A_K^{v} \arrow[d]\arrow[r]  & \A_K^{v,c} \arrow[d] \arrow[r] & \A_K^\emptyset \arrow[d] \arrow[r] & 0\\
	\cF_L^{\theta(v)} \colon 0\arrow{r}&	\A_L^{\theta(v)} \arrow[r]  & \A_L^{\theta(v),c}\arrow[r] & \A_L^\emptyset \arrow[r] & 0\nospacepunct{.}
 \end{tikzcd}\end{equation}
Note that $\A_K^\emptyset = C^*(\Gamma_K)^{\oplus \mu_K}$. 
Similarly to \cite[Proposition~5.4]{BT}, by composing $\alpha$ with $\tau_\chi$ for some $\chi \in \widehat{L^*}$, where $\tau \colon \widehat{L^*} \acts \A_L$ is the dual action, we obtain an isomorphism $\alpha^v_* \colon \K_*(\B_K^v) \to \K_*(\B_L^{\theta(v)})$ for each $v \in \Sigma_K$ and a *-isomorphism $\alpha^\emptyset \colon C^*(\Gamma_K) \to C^*(\Gamma_L)$ such that the following diagram commutes:
 \begin{equation} \begin{tikzcd}
	\K_*(C^*(\Gamma_K)) \arrow[r, "\partial_K^v"] \arrow[d, "\alpha^\emptyset_*"] & \K_*(\B_K^v) \arrow[d, "\alpha^v_*"] \\
	\K_*(C^*(\Gamma_L)) \arrow[r, "\partial_L^{\theta(v)}"] & \K_*(\B_L^{\theta(v)})\nospacepunct{.}
 \end{tikzcd}\end{equation}
Note that the bijection $\Sigma_L \to \Sigma_L$ induced by $\tau_\chi$ is the identity map. 
Hence, the first claim holds by Theorem~\ref{thm:main}. 
It follows that $\alpha(P_{K,\infty})=P_{L,\infty}$, so the second claim holds by \cite[Theorem~1.1]{BT}. 
\end{proof}

\begin{remark}
\label{rmk:ktheory}
The full strength of Theorem~\ref{thm:main} is not needed to derive the second claim in Corollary~\ref{cor:main}. Indeed, we can determine whether or not $v$ is a finite place from properties of $\K_0(\B_K^v)$ as follows. 
Suppose $v$ is a finite place corresponding to a prime $\p$, and let $p$ be the rational prime lying under $\p$. Then, $\B_K^v\cong \Kz\otimes B_K^v$, where $B_K^v$ is the ``unital part'' of $\B_K^v$ (see \cite[Definition~2.8]{BT}). The C*-algebra $B_K^v$ falls into the general setting of \cite[\S~3.3]{Tak2}. C*-algebras from \cite[\S~3.3]{Tak2} all have a unique trace, and the image of $\K_0$ under this trace is a direct summand in their $\K_0$-group (cf. the proof of \cite[Proposition~3.12]{Tak2}). Hence, $\K_0(B_K^v)$ contains $\Zz[1/p]$ as a direct summand (see \cite[\S~9.1]{BT}).
On the other hand, if $v$ is infinite, then by the proof of Proposition~\ref{prop:bijinj}, $\K_0(\B_K^v)\cong \K_1(C^*(\Gamma_K))$ is free abelian.
\end{remark}

\end{document}